\RequirePackage[l2tabu, orthodox]{nag}
\documentclass{article}

\usepackage{amsmath,amsthm,amssymb}
\usepackage{bbm,booktabs}
\usepackage{color}
\usepackage{enumerate}
\usepackage{enumitem}
\usepackage[acronym]{glossaries}
\usepackage{graphicx}
\usepackage{multirow}
\usepackage{nicefrac}
\usepackage[percent]{overpic}
\usepackage{pgfplots}
\usepgfplotslibrary{groupplots}
\usepackage{rotating} 
\usepackage{setspace}
\usepackage{subfigure}
\usepackage{tikz}
\usepackage[colorinlistoftodos,textsize=scriptsize]{todonotes}
\usepackage{units}

\usepackage{epstopdf} 
\usepackage{hyperref}

\definecolor{MyDarkBlue}{rgb}{0,0.08,0.50}
\definecolor{BrickRed}{rgb}{0.65,0.08,0}

\hypersetup{
colorlinks=true,       		
    linkcolor=MyDarkBlue,   
    citecolor=BrickRed,     
    filecolor=red,      	
    urlcolor=cyan           
}

\usepackage{dsfont} 

\renewcommand{\d}[1]{\ensuremath{\operatorname{d}\!{#1}}}
\newcommand{\e}[1]{ {\mathrm{e}}^{ #1 } }
\newcommand{\im}{\mathrm{i}}

\newcommand{\indicator}[1]{ \mathds{1} [ #1 ] }

\newcommand{\process}[2]{ \{ #1 \}_{ #2 } }
\newcommand{\smallO}[1]{ o(#1) }
\newcommand{\bigO}[1]{ O(#1) }
\newcommand{\bigObig}[1]{ O\Bigl(#1\Bigr) }

\newcommand{\naturalNumbersPlus}{ \mathbb{N}_{+} }
\newcommand{\naturalNumbersZero}{ \mathbb{N}_{0} }
\newcommand{\realNumbers}{ \mathbb{R} }

\newcommand{\criticalpoint}[1]{  #1^{\textnormal{opt}} }

\newcommand{\normalizationConstant}{Z}

\newcommand{\refFigure}[1]{{\textrm{Figure~\ref{#1}}}}
\newcommand{\refTable}[1]{{\textrm{Table~\ref{#1}}}}
\newcommand{\refEquation}[1]{{\textrm{\eqref{#1}}}}

\newcommand{\refTheorem}[1]{{\textrm{Theorem~\ref{#1}}}}
\newcommand{\refCorollary}[1]{{\textrm{Corollary~\ref{#1}}}}

\newcommand{\refLemma}[1]{{\textrm{Lemma~\ref{#1}}}}

\newcommand{\refSection}[1]{{\textrm{\S\ref{#1}}}}
\newcommand{\refAppendixSection}[1]{{\textrm{\ref{#1}}}}

\newcommand{\QuodEratDemonstrandum}{\hfill \ensuremath{\Box}}

\newtheorem{theorem}{Theorem}

\newtheorem{lemma}{Lemma}
\newtheorem{corollary}{Corollary}

\usepackage{authblk}

\begin{document}

\title{Optimality gaps in asymptotic dimensioning of many-server systems}

\date{}

\author{Jaron Sanders\thanks{jaron.sanders@tue.nl}}

\author{S.C.\ Borst\thanks{s.c.borst@tue.nl}}

\author{A.J.E.M.\ Janssen\thanks{a.j.e.m.janssen@tue.nl}}

\author{J.S.H.\ van Leeuwaarden\thanks{j.s.h.v.leeuwaarden@tue.nl}}


\affil{Department of Mathematics \& Computer Science, Eindhoven University of Technology, P.O.\ Box 513, 5600 MB Eindhoven, The Netherlands}

\newacronym{QED}{QED}{Quality-and-Efficiency-Driven}
\newacronym{OU}{OU}{Ornstein-Uhlenbeck}
\newacronym{EM}{EM}{Euler-Maclaurin}

\maketitle

\begin{abstract}
The \gls{QED} regime provides a basis for solving asymptotic dimensioning problems that trade off revenue, costs and service quality. We derive bounds for the \emph{optimality gaps} that capture the differences between the true optimum and the asymptotic optimum based on the QED approximations. Our bounds generalize earlier results for classical many-server systems. We also apply our bounds to a many-server system with threshold control.
\vspace{1em}


\end{abstract}

\glsreset{QED}
\section{Introduction}

The theory of square-root staffing in many-server systems ranks among the most celebrated principles in applied probability. The general idea behind square-root staffing is as follows: a finite server system is modeled as a system in heavy traffic, where the number of servers $s$ is large, whereas at the same time, the system is critically loaded. Under Markovian assumptions, and denoting the load on the system by $\lambda$, this can be achieved by setting $s = \lambda +\beta \sqrt{\lambda}$ and letting $\lambda \to \infty$ while keeping $\beta > 0$ fixed, or alternatively setting $\lambda = s - \gamma \sqrt{s}$ and letting $s \to \infty$ while keeping $\gamma > 0$ fixed. In both cases, the system reaches the desirable {\gls{QED} regime}.

The \gls{QED} regime refers to mathematically defined conditions in which both customers and system operators benefit from the advantages that come with systems that operate efficiently at large scale, which is particularly relevant for systems in e.g.\ health care, cloud computing, and customer services. Such conditions manifest themselves in a low delay probability and negligible mean delay, despite the fact that the system utilization is high. Properties of this sort can be proven rigorously for systems such as the $M/M/s$ queue by establishing stochastic-process limits under the aforementioned \gls{QED} scalings \cite{halfin_heavy-traffic_1981}. The \gls{QED} regime also creates a natural environment for solving dimensioning problems that achieve an acceptable trade-off between service quality and capacity. Quality is usually formulated in terms of some target service level. Take for instance the probability that an arriving customer experiences delay. The target could be to keep the delay probability below some value $\epsilon \in (0,1)$. The smaller $\epsilon$, the better the offered quality of service. Once the target service level is set, the objective from the operator's perspective is to determine the highest load $\lambda$ such that the target $\epsilon$ is still met.

For the $M/M/s$ queue, it was shown by Borst et al.\ \cite{borst_dimensioning_2004} that such dimensioning procedures combined with \gls{QED} approximations have certain asymptotic optimality properties. To illustrate this, consider the case of linear costs, i.e.\ waiting cost are $b$ per customer per unit time, and staffing cost are $c$ per server per unit time. Denoting the total cost by $K_\lambda(s)$, it can be shown that when $s = \lambda + \beta \sqrt{\lambda}$ and $\beta > 0$,
\begin{equation}
K_\lambda(s)
= b \lambda \frac{ C_\lambda(s) }{ s - \lambda } + c s
= c \lambda + \sqrt{\lambda} \Bigl( c \beta + \frac{b}{\beta} C_\lambda(s) \Bigr)
\label{eqn:Erlang_C_cost_structure}
\end{equation}
with $C_\lambda(s)$ the delay probability in the $M/M/s$ queue. The first term $c \lambda$ represents the cost of the minimally required capacity $\lambda$, while the second term gathers the cost factors that are all $\bigO{ \sqrt{\lambda} }$. Halfin and Whitt \cite{halfin_heavy-traffic_1981} showed that in the \gls{QED} regime $C_\lambda(s)$ converges to a nondegenerate limit $C_0(\beta)\in (0,1)$, so that in the \gls{QED} regime one only needs to determine $\beta_0 = \arg \min_{\beta} \{ c \beta + b C_0(\beta) / \beta \}$, and then set $s_0 = [ \lambda + \beta_0 \sqrt{\lambda} ]$ as an approximation for the optimal number of servers $\criticalpoint{s} = \arg \min_{s} \{ K_\lambda(s) \}$. Borst et al.\ \cite{borst_dimensioning_2004} called this procedure \textit{asymptotic dimensioning}.

Based on the \gls{QED} limiting regime, one expects that such approximate solutions are accurate for large relative loads $\lambda$. For the \emph{optimality gaps} $| s_0 - \criticalpoint{s} |$ and $| K_\lambda(s_0) - K_\lambda(\criticalpoint{s}) |$, i.e.\ inaccuracies that arise from the fact that the actual system is of finite size, Borst et al.\ \cite{borst_dimensioning_2004} showed through numerical experiments that the approximation $s_0$ performs exceptionally well in almost all circumstances, even when systems are only moderately sized. A rigorous underpinning for these observations was provided by Janssen et al.\ \cite{janssen_refining_2011}, who used \emph{refined} \gls{QED} approximations to quantify the optimality gaps. The delay probability, for instance, was shown to behave as $C_0(\beta) + C_1(\beta) / \sqrt{\lambda} + \bigO{ \lambda^{-1} }$, which in turn was used to estimate the optimality gaps for the dimensioning problem in \refEquation{eqn:Erlang_C_cost_structure}. Zhang et al.\ \cite{zhang_staffing_2012} obtained similar results for optimality gaps in the context of the $M/M/s+M$ queue, in which customers may abandon before receiving service.

Motivated by the results in \cite{janssen_refining_2011,zhang_staffing_2012}, Randhawa \cite{randhawa_optimality_2012} took a more abstract approach to quantify optimality gaps of asymptotic dimensioning problems. He showed under general assumptions that when the approximation to the objective function is accurate up to $\bigO{1}$, the prescriptions that are derived from this approximation are $\smallO{1}$-optimal. The optimality gap thus becomes zero asymptotically. This general setup was shown in \cite{randhawa_optimality_2012} to apply to the $M/M/s$ queues in the \gls{QED} regime, which confirmed and sharpened the results on the optimality gaps in \cite{janssen_refining_2011,zhang_staffing_2012} by implying that $| K_\lambda(s_0) - K_\lambda(\criticalpoint{s}) | = \smallO{1}$. The abstract framework in \cite{randhawa_optimality_2012}, however, can only be applied if refined approximations as in \cite{janssen_refining_2011,zhang_staffing_2012} are available.

Such refined approximations were recently developed in \cite{janssen_scaled_2013,sanders_optimal_2014} for a broad class of many-server systems operating in the \gls{QED} regime with $\lambda = s - \gamma \sqrt{s}$, and equipped with an admission control policy and a revenue structure. For a wide range of performance metrics, $M_s(\gamma)$ say, these refinements are of the form $M_s(\gamma) = M_0(\gamma) + M_1(\gamma) / \sqrt{s} + \cdots$. The method in \cite{janssen_scaled_2013,sanders_optimal_2014} can deliver as many higher-order terms as needed, and generate all the refinements obtained in \cite{janssen_refining_2011,zhang_staffing_2012,randhawa_optimality_2012}.

In the present paper, we demonstrate how the results in \cite{janssen_scaled_2013,sanders_optimal_2014} can be leveraged to determine the optimality gaps of novel asymptotic dimensioning problems for a large class of many-server systems. Our main result (\refTheorem{thm:Objective_gap}) provides generic bounds for the optimality gaps that become sharper when more terms in the \gls{QED} expansion for $M_s(\gamma)$ are included.

\section{Model description}
\label{sec:Preliminaries}

\subsection{Service systems with admission control and revenues}
\label{sec:Service_systems_with_admission_control_and_revenues}

We consider many-server systems with $s$ parallel servers, to which customers arrive according to a Poisson process with rate $\lambda$. Every customer requires an exponentially distributed service time with mean one. If a customer arrives and finds $k-s \geq 0$ customers waiting, the customer is allowed to join the queue with probability $p_s(k-s)$ and is rejected with probability $1 - p_s(k-s)$. The total number of customers in the system evolves as a birth--death process $\process{ X_s(t) }{t \geq 0}$ and has a stationary distribution
\begin{equation}
\pi_s(k)
=
\begin{cases}
\normalizationConstant^{-1}, & k = 0, \\ 
\normalizationConstant^{-1} \frac{(s\rho)^k}{k!}, & k = 1, 2, \ldots, s, \\
\normalizationConstant^{-1} \frac{s^s \rho^k}{s!} \prod_{i=0}^{k-s-1} p_s(i), & k = s+1, s+2, \ldots, \\
\end{cases}
\label{eqn:Equilibrium_distribution}
\end{equation}
where $\rho = \lambda / s$, $\normalizationConstant = \sum_{k=0}^s (s\rho)^k / k! + ( (s\rho)^s / s! ) F_s(\rho)$, and $F_s(\rho) = \sum_{n=0}^\infty p_s(0) \cdot \ldots \cdot p_s(n) \rho^{n+1}$. The stationary distribution in \refEquation{eqn:Equilibrium_distribution} exists if and only if the relative load $\rho$ and the \emph{admission control policy} $\{ p_s(k) \}_{k \in \naturalNumbersZero}$ are such that $F_s(\rho) < \infty$.

Next, we assume that the system generates revenue at rate $r_s(k) \in \realNumbers$ when there are $k$ customers in the system. The sequence $\{ r_s(k) \}_{k \in \naturalNumbersZero}$ will be called the \emph{revenue structure}. The stationary rate at which the system generates revenue is then given by
\begin{equation}
R_s(\gamma)
= \sum_{k=0}^\infty r_s(k) \pi_s(k),
\label{eqn:Weight_function_Rs}
\end{equation}
which depends via the equilibrium distribution on the admission control policy. Ref.\ \cite{sanders_optimal_2014} discusses the problem of maximizing the revenue rate over the set of all admission control policies.

One advantage of considering general admission control policies and revenue structures is that one can study different service systems and steady-state performance measures through one unifying framework. For example, the equilibrium behavior of the canonical $M/M/s/s$, $M/M/s$, and $M/M/s+M$ systems can be recovered by choosing $p_s(k-s) = 0$, $p_s(k-s) = 1$, and $p_s(k-s) = 1/(1 + (k-s+1)\theta/s)$, respectively. Here, $\theta$ corresponds to the rate at which waiting customers abandon from the $M/M/s+M$ system. Similarly, the delay probability $D_s(\gamma) = \sum_{k=s}^\infty \pi_s(k)$ can be recovered by setting $r_s(k) = \indicator{ k \geq s }$, the mean queue length $Q_s(\gamma) = \sum_{k=s}^\infty (k-s) \pi_s(k)$ is recovered when considering $r_s(k) = (k-s) \indicator{ k \geq s }$, and the average number of idle servers $I_s(\gamma) = \sum_{k=0}^{s-1} (s-k) \pi_s(k)$ follows from $r_s(k) = (s-k) \indicator{ k < s }$. 

As a primary example we will consider a scenario where besides the
waiting cost $b > 0$ incurred per customer per unit time, a fee $a > 0$ is
received for every served customer, and a penalty $d \geq 0$ is
imposed for rejecting a customer.
The latter cost accounts for the degree of revenue loss from the
admission control policy. Denoting by $D_s^R(\gamma) = \sum_{k=s}^\infty (1 - p_s(k-s)) \pi_s(k)$ the probability that an arriving
customer is rejected, and by $W_s(\gamma) = \sum_{k=s}^\infty ( ( k - s +1 ) / s ) p_s(k-s) \pi_s(k)$ the expected waiting time of an arriving customer, the total system revenue rate is given by
\begin{align}
&
R_s(\gamma)
= a \lambda ( 1 - D_s^R(\gamma) ) - b \lambda W_s(\gamma) - d \lambda D_s^R(\gamma).
\label{eqn:Economics_optimization_problem} 
\end{align}
By virtue of Little's law $\lambda W_s(\gamma) = Q_s(\gamma)$ and $\lambda (1 -  D_s^R(\gamma)) = s - I_s(\gamma)$, and since $\lambda = s - \gamma \sqrt{s}$, the revenue rate can equivalently be expressed as
\begin{align}
R_s(\gamma)
&
= a s + d \gamma \sqrt{s} - (a+d) I_s(\gamma) - b Q_s(\gamma).
\label{eqn:Rs_from_intuitive_derivation}
\end{align}
This scenario therefore corresponds to the revenue structure
\begin{equation}
r_s(k)
= 
\begin{cases}
a k + d \gamma \sqrt{s} - d (s-k) & k < s, \\
a s + d \gamma \sqrt{s} - b (k-s), & k \geq s. \\
\end{cases}
\label{eqn:Revenue_structure_of_the_economics_optimization_problem}
\end{equation}

\subsection{\texorpdfstring{\gls{QED}}{QED} scaling and refinements}
\label{sec:Critical_scaling_and_QED_refinements}


We now discuss how to apply the
\gls{QED} scaling to obtain an asymptotic expansion for $R_s(\gamma)$ for general revenue structures $\{ r_s(k) \}_{k \in \naturalNumbersZero}$, which we will exploit in \refSection{sec:Generalized_square_root_staffing} to
characterize the asymptotic optimality gap.
We impose the following three conditions throughout this paper:
\begin{enumerate}
\item[(i)] The arrival rate and system size are coupled via the scaling $\lambda = s - \gamma \sqrt{s}$;
\item[(ii)] $p_s(0) \cdot \ldots \cdot p_s(n) \to f( (n+1)/\sqrt{s} )$ where $f(x)$ is either a continuous, nonincreasing function, or $f(x) = \indicator{ x \leq \eta }$;
\item[(iii)] There exist sequences $\{ n_s \}_{s \in \naturalNumbersPlus}$, $\{ q_s \}_{s \in \naturalNumbersPlus}$ with $q_s > 0$, and a continuous function $r(x)$ that satisfy the scaling condition $(r_s(k) - n_s) / q_s \to r( (k-s) / \sqrt{s} )$.
\end{enumerate}

It is proven in \cite{janssen_scaled_2013,sanders_optimal_2014} that $\lim_{s \to \infty} ( R_s(\gamma) \allowbreak - n_s ) / q_s = R_0(\gamma)$ under conditions (i)--(iii), with
\begin{equation}
R_0(\gamma) = \frac{ \int_{-\infty}^0 r(x) \e{- \frac{1}{2} x^2 - \gamma x } \d{x} + \int_0^\infty r(x) f(x) \e{-\gamma x} \d{x} }{ \frac{\Phi(\gamma)}{\phi(\gamma)} + \int_0^\infty f(x) \e{-\gamma x} \d{x} }.
\label{eqn:Limit_of_weight_function_Rs}
\end{equation}
Here, $\Phi$ and $\phi$ denote the cumulative distribution function and probability density function of the standard normal distribution. This asymptotic characterization of the revenue is leveraged in \cite{sanders_optimal_2014} to prove that for many revenue structures there exists an optimal admission control policy with a threshold structure. 

Moreover, the method used to obtain \refEquation{eqn:Limit_of_weight_function_Rs} in \cite{janssen_scaled_2013,sanders_optimal_2014} can be extended to derive an asymptotic expansion of the form
\begin{equation}
R_s(\gamma)
= n_s + q_s \Bigl( \sum_{i=0}^j \frac{R_i(\gamma)}{s^{i/2}} + \bigObig{ \frac{1}{s^{(j+1)/2}} } \Bigr),
\label{eqn:Expansion_for_Rs}
\end{equation}
which is shown in \refAppendixSection{sec:Expansions}. We also provide closed-form expressions for the first two terms $R_0(\gamma)$ and $R_1(\gamma)$ for arbitrary $f(x)$ and $r(x)$. The asymptotic expansion in \refEquation{eqn:Expansion_for_Rs} is a crucial ingredient for determining the optimality gaps.


Let us discuss the asymptotic expansion in the context of \refEquation{eqn:Rs_from_intuitive_derivation}. Denoting $n_s = as$ and extracting a term $q_s = \sqrt{s}$ yields $R_s(\gamma) = n_s +  \sqrt{s} \hat{R}_s(\gamma)$ with
\begin{equation}
\hat{R}_s(\gamma) = d \gamma - (a+d) \frac{ I_s(\gamma) }{\sqrt{s}} - b \frac{ Q_s(\gamma) }{\sqrt{s}}.
\label{eqn:Expression_for_Rshat}
\end{equation} 
Since our goal is to maximize $R_s(\gamma)$ over $\gamma$, and because the term $n_s$ is constant and independent of $\gamma$, we only need to focus on the maximization of $\hat{R}_s(\gamma)$.

Recall \refEquation{eqn:Revenue_structure_of_the_economics_optimization_problem} and note that the limiting revenue structure for the objective function in \refEquation{eqn:Expression_for_Rshat} is given by
$r(x) = (a+d)x + d \gamma$ for $x < 0$, and $r(x) = -bx + d \gamma$ for $x \geq 0$.
With an admission control policy $f(x) = \indicator{ x \leq \eta }$ where $\eta \geq 0$ denotes the admission threshold, it follows from \refEquation{eqn:Limit_of_weight_function_Rs} that
\begin{equation}
\lim_{s \to \infty} \hat{R}_s(\gamma)
= \hat{R}_0(\gamma)
= d \gamma - 
\frac{ (a+d) \bigl( 1 + \gamma \frac{\Phi(\gamma)}{\phi(\gamma)} \bigr) +
b \frac{1 - (1+\gamma \eta) \e{-\gamma \eta} }{\gamma^2} }
{\frac{\Phi(\gamma)}{\phi(\gamma)} + \frac{1 - \e{-\gamma \eta}}{\gamma} }.
\label{eqn:Asymptotic_limit_of_economics_optimizaton_problem}
\end{equation}
We prove in \refAppendixSection{sec:Expansions} that $\lim_{s \to \infty} \sqrt{s} ( \hat{R}_s(\gamma) - \hat{R}_0(\gamma) ) = \hat{R}_1(\gamma)$ and provide an explicit expression for $\hat{R}_1(\gamma)$. We then have both a first-order approximation $\hat{R}_0(\gamma)$ and a second-order approximation $\hat{R}_0(\gamma) + \hat{R}_1(\gamma) / \sqrt{s}$ for the objective function $\hat{R}_s(\gamma)$.

\section{General revenue maximization}
\label{sec:Generalized_square_root_staffing}

For general objective functions \refEquation{eqn:Weight_function_Rs}, we now aim for solving the dimensioning problem
\begin{equation}
\max_{\gamma \in \Gamma } \{ R_s(\gamma) \},
\label{eqn:Minimization_problem}
\end{equation}
where we assume that $\Gamma = [\gamma_l, \gamma_r]$ is a compact interval contained in $( \gamma^{\min}_s, \infty )$, with $\gamma^{\min}_s = \inf \{ \gamma \in \realNumbers | F_s(\rho) < \infty \}$. Denote the exact solution by
\begin{equation}
\criticalpoint{\gamma}_s
= \arg \max_{\gamma \in \Gamma } \{ R_s(\gamma) \}.
\label{eqn:Exact_solution}
\end{equation}

We assume that an asymptotic expansion of the form \refEquation{eqn:Expansion_for_Rs} is available for $R_s(\gamma)$ and its derivative ${R_s}'(\gamma)$, which we can then use to approximate the objective function. Hence we will consider
\begin{equation}
\gamma_{j,s}
= \arg \max_{\gamma \in \Gamma } \Bigl\{ R_0(\gamma) + \ldots + \frac{R_j(\gamma)}{s^{j/2}} \Bigr\}
\end{equation}
as approximations for the exact solution $\criticalpoint{\gamma}_s$, which should be increasingly better for larger $j$ and/or $s$. Note that $\gamma_{0,s} = \gamma_0$ is independent of $s$.

Denoting the $i$th derivative of a function $g(x)$ by $g^{(i)}(x)$, we assume also that ${R_0}^{(k)}(\gamma), \ldots, {R_j}^{(k)}(\gamma)$ are bounded on $\Gamma$ for $k = 0, 1, 2$, and that ${R_{j+1}}^{(l)}(\gamma)$ is bounded on $\Gamma$ for $l = 0,1$. We furthermore assume that both the first-order optimizer $\gamma_0$ and the exact optimizer $\criticalpoint{\gamma}_s$ exist, are unique and lie in the interior of $\Gamma$, and that ${R_0}(\gamma)$ is strictly concave on $\Gamma$ and has a continuous derivative ${R_0}''(\gamma)$ on $\Gamma$. Finally, we assume that $f(x)$ is such that
\begin{equation}
\lim_{\gamma \downarrow \gamma^{\min}} \int_0^\infty f(x) \e{-\gamma x} \d{x} 
= \infty,
\label{eqn:Assumption_on_Lf}
\end{equation}
where $\gamma^{\min} = \inf \{ \gamma \in \realNumbers | \int_0^\infty f(x) \e{-\gamma x} \d{x} < \infty \}$. Ref.\ \cite{janssen_scaled_2013} discusses under which conditions assumption \refEquation{eqn:Assumption_on_Lf} is satisfied, and it is satisfied for instance for any admission control policy $f(x) = \indicator{ x \leq \eta }$ with admission threshold $\eta > 0$, since $\int_0^\infty f(x) \e{-\gamma x} \d{x} \allowbreak = (1-\e{-\gamma \eta})/\gamma$ and $\gamma^{\min} = -\infty$.

\subsection{Optimality gaps}
\label{sec:Optimality_gaps}

We now derive the optimality gaps for the general optimization problem in \refEquation{eqn:Minimization_problem}. \refTheorem{thm:Objective_gap} formalizes that the approximating solutions $\gamma_{j,s}$ are asymptotically optimal through bounds for the optimality gaps, and that an approximation of order $j$ yields a gap decay of order $j+1$. With minor modifications to the proof, the result also applies to minimization problems of the form $\min_{\gamma \in \Gamma} \{ R_s(\gamma) \}$.

\begin{theorem}
\label{thm:Objective_gap}
For $j = 0, 1, \ldots$, there exist constants $M_j$ and $K_j$ independent of $s$ and $s_j \in \naturalNumbersPlus$ such that for all $s \geq s_j$,
\begin{equation}
| R_s( \criticalpoint{\gamma}_s ) - R_s( \gamma_{j,s} ) |
\leq \frac{q_s M_j}{ s^{(j+1)/2} },
\quad
| \gamma_{j,s} - \criticalpoint{\gamma}_s |  
\leq \frac{K_j}{ s^{(j+1)/2} }.
\label{eqn:Optimality_gaps_in_main_theorem}
\end{equation}
\end{theorem}

\paragraph{Proof} 
First, we will prove a monotonicity result, as well as the existence of optimizers.

\begin{lemma}
\label{lem:Monotonicity_and_existence}
There is an $s_0 \in \naturalNumbersPlus$ such that for all $s \geq s_0$, the function $R_0(\gamma) + \sum_{i=1}^j R_i(\gamma) / s^{i/2}$ has a unique optimizer $\gamma_{j,s} \in \Gamma$ and a strictly decreasing derivative ${R_0}'(\gamma) + \sum_{i=1}^j {R_i}'(\gamma) / s^{i/2}$.
\end{lemma}

\begin{proof}
Recall that $\gamma_0$ lies in the interior of $\Gamma$ and that ${R_0}'(\gamma)$ is strictly decreasing on $\Gamma$ by assumption, which implies that ${R_0}'(\gamma_l) > 0$. We seek $s_1 \in \naturalNumbersPlus$ such that for all $s \geq s_1$, ${R_0}'(\gamma_l) \allowbreak + \allowbreak \sum_{i=1}^j {R_i}'(\gamma_l) / s^{i/2} > 0$. Note next that
$(1/\sqrt{s}) \sum_{i=1}^j | {R_i}'(\gamma_l) | \allowbreak \geq - \sum_{i=1}^j {R_i}'(\gamma_l) / s^{i/2}$ for $s \in \naturalNumbersPlus$.
For all $s \in \naturalNumbersPlus$ for which ${R_0}'(\gamma_l) \allowbreak > \allowbreak (1/\sqrt{s}) \sum_{i=1}^j | {R_i}'(\gamma_l) |$, we have thus consequently that ${R_0}'(\gamma_l) \allowbreak > - \sum_{i=1}^j {R_i}'(\gamma_l) / \allowbreak s^{i/2} $. We therefore pick
$
s_1 = \lceil ( \sum_{i=1}^j | {R_i}'(\gamma_l) | / \allowbreak {R_0}'(\gamma_l) )^2 \rceil
$
to ensure that for all $s \geq s_1$, ${R_0}'(\gamma_l) + \sum_{i=1}^j {R_i}'(\gamma_l) / s^{i/2} \allowbreak > 0$. A similar result holds at $\gamma = \gamma_r$, i.e.\ we have ${R_0}'(\gamma_r) < 0$, and thus
$
s_2 = \lceil ( \sum_{i=1}^j | {R_i}'(\gamma_r) | \allowbreak / \allowbreak {R_0}'(\gamma_r) )^2 \rceil
$
is such that for all $s \geq s_2$, ${R_0}'(\gamma_r) + \sum_{i=1}^j {R_i}'(\gamma_r) / s^{i/2} < 0$. The function $R_0(\gamma) + \sum_{i=1}^j R_i(\gamma) / s^{i/2}$ thus has a unique optimizer $\gamma_{j,s} \in \Gamma$.

Finally we turn to proving the monotonicity property of ${R_0}'(\gamma) \allowbreak + \sum_{i=1}^j {R_i}'(\gamma) / \allowbreak s^{i/2}$. By assumption, ${R_0}''(\gamma) < 0$ for all $\gamma \in [\gamma_l,\gamma_r]$. Similar to before, set
$
s_3 = \max_{\gamma \in [\gamma_l,\gamma_r]} \lceil ( \sum_{i=1}^j | {R_i}''(\gamma) | / \allowbreak {R_0}''(\gamma) )^2 \rceil,
$
and conclude that ${R_0}''(\gamma) + \sum_{i=1}^j {R_i}''(\gamma) / \allowbreak s^{i/2} < 0$ for all $s \geq s_3$ and all $\gamma \in [\gamma_l,\gamma_r]$. Finish the proof by setting $s_0 = \max \{ s_1, s_2, s_3 \}$, and by noting that $s_0$ is bounded.
\end{proof}

Recall that the unique optimizer $\criticalpoint{\gamma}_s$ exists, and lies in the interior of $\Gamma$. Because $\criticalpoint{\gamma}_s$ maximizes $R_s(\gamma)$, we have therefore by suboptimality that
\begin{align}
&
0 
\leq R_s(\criticalpoint{\gamma}_s) - R_s(\gamma_{j,s})
= \Bigl[ q_s \sum_{i=0}^j \frac{R_i(\gamma_{j,s})}{s^{i/2}} - R_s(\gamma_{j,s}) \Bigr] 
\\ &
- \Bigl[ q_s \sum_{i=0}^j \frac{R_i(\criticalpoint{\gamma}_s)}{s^{i/2}} - R_s(\criticalpoint{\gamma}_s) \Bigr] 
+ q_s \Bigl[ \sum_{i=0}^j \frac{R_i(\criticalpoint{\gamma}_s)}{s^{i/2}} - \sum_{i=0}^j \frac{R_i(\gamma_{j,s})}{s^{i/2}} \Bigr],
\nonumber
\end{align}
and subsequently by expansion \refEquation{eqn:Expansion_for_Rs} that
\begin{align}
0 & \leq R_s(\criticalpoint{\gamma}_s) - R_s(\gamma_{j,s})
= q_s \Bigl[ \sum_{i=0}^j \frac{R_i(\criticalpoint{\gamma}_s)}{s^{i/2}} - \sum_{i=0}^j \frac{R_i(\gamma_{j,s})}{s^{i/2}} \Bigr] 
\nonumber \\ &
\phantom{\leq} + q_s \frac{R_{j+1}(\criticalpoint{\gamma}_s) - R_{j+1}(\gamma_{j,s})}{s^{(j+1)/2}} 
+ \bigObig{ \frac{q_s}{s^{(j+2)/2}} }.
\end{align}
Since $\gamma_{j,s}$ maximizes $\sum_{i=0}^j R_i(\gamma) / s^{i/2}$, we have by suboptimality that the term within square brackets is negative, i.e.\
\begin{equation}
\sum_{i=0}^j \frac{R_i(\criticalpoint{\gamma}_s)}{s^{i/2}} - \sum_{i=0}^j \frac{R_i(\gamma_{j,s})}{s^{i/2}}
\leq 0
\end{equation}
for all $s \in \naturalNumbersZero$. Therefore, since $q_s > 0$,
\begin{equation}
0 
\leq R_s(\criticalpoint{\gamma}_s) - R_s(\gamma_{j,s})
\leq q_s \frac{R_{j+1}(\criticalpoint{\gamma}_s) - R_{j+1}(\gamma_{j,s})}{s^{(j+1)/2}} + \bigObig{ \frac{q_s}{s^{(j+2)/2}} }.
\end{equation}
Since $R_{j+1}(\gamma)$ is bounded on $\Gamma$ by assumption, the first claim in \refEquation{eqn:Optimality_gaps_in_main_theorem} follows. 

\begin{corollary}
\label{cor:Bound_on_derivative_of_expansion_at_gamma_opt}
There exist constants ${M_j}' > 0$ independent of $s$ and ${s_j}' \in \naturalNumbersPlus$ such that for all $s \geq {s_j}'$,
\begin{equation}
\Bigl| \sum_{i=0}^j \frac{{R_i}'(\criticalpoint{\gamma}_s)}{s^{i/2}} - \sum_{i=0}^j \frac{{R_i}'(\gamma_{j,s})}{s^{i/2}} \Bigr|
= \Bigl| \sum_{i=0}^j \frac{{R_i}'(\criticalpoint{\gamma}_s)}{s^{i/2}} \Bigr|
\leq \frac{{M_j}'}{s^{(j+1)/2}}. \label{eqn:Bound_on_derivative_of_expansion_at_gamma_opt}
\end{equation}
\end{corollary}

\begin{proof}
Note that $\gamma_{j,s}$ is the optimizer of $\sum_{i=0}^j R_i(\gamma) / s^{i/2}$, and that therefore $\sum_{i=0}^j {R_i}'(\gamma_{j,s}) / s^{i/2} = 0$, which proves the leftmost equality. 

Next, we examine the asymptotic expansion of the derivative of $R_s(\gamma)$, which we have assumed is available and of the form \refEquation{eqn:Expansion_for_Rs}. It follows that
\begin{align}
q_s \sum_{i=0}^j \frac{{R_i}'(\criticalpoint{\gamma}_s)}{s^{i/2}} 
&
= R_s'(\criticalpoint{\gamma}_s) - q_s \frac{{R_{j+1}}'(\criticalpoint{\gamma}_s)}{s^{(j+1)/2}} + \bigObig{ \frac{q_s}{s^{(j+2)/2}} }
\nonumber \\ &
\overset{\textrm{(i)}}= - q_s \frac{{R_{j+1}}'(\criticalpoint{\gamma}_s)}{s^{(j+1)/2}} + \bigObig{ \frac{q_s}{s^{(j+2)/2}} },
\end{align}
since (i) $\criticalpoint{\gamma}_s$ optimizes $R_s(\gamma)$. Now \refEquation{eqn:Bound_on_derivative_of_expansion_at_gamma_opt} follows since ${R_{j+1}}'(\gamma)$ is bounded on $\Gamma$.
\end{proof}

We are now ready to establish the second claim in \refEquation{eqn:Optimality_gaps_in_main_theorem}. Recall that for all $j \in \naturalNumbersPlus$, and sufficiently large $s$, the function $\sum_{i=0}^j {R_i}'(\gamma) / s^{(i/2)}$ is strictly decreasing in $[\gamma_l,\gamma_r]$, see \refLemma{lem:Monotonicity_and_existence}. Note also that $\gamma_0, \criticalpoint{\gamma}_s, \gamma_{j,s} \in [\gamma_l,\gamma_r]$. The mean value theorem implies then that 
\begin{equation}
\Bigl| \sum_{i=0}^j \frac{{R_i}'(\criticalpoint{\gamma}_s)}{s^{(i/2)}} - \sum_{i=0}^j \frac{{R_i}'(\gamma_{j,s})}{s^{(i/2)}} \Bigr|
\geq m_{j,s} | \criticalpoint{\gamma}_s - \gamma_{j,s} |
\end{equation}
with
\begin{equation}
m_{j,s} = - \max_{\gamma \in [\gamma_l,\gamma_r]} \Bigl\{ {R_0}''(\gamma) + \sum_{i=1}^j \frac{{R_i}''(\gamma)}{s^{i/2}} \Bigr\},
\end{equation}
where we have also used that the function $\sum_{i=0}^j R_i(\gamma) / s^{(i/2)}$ is optimized by $\gamma_{j,s}$. Combining with \refCorollary{cor:Bound_on_derivative_of_expansion_at_gamma_opt}, it follows that
\begin{equation}
\frac{{M_j}'}{s^{(j+1)/2}} 
\geq \Bigl| \sum_{i=0}^j \frac{{R_i}'(\criticalpoint{\gamma}_s)}{s^{i/2}} - \sum_{i=0}^j \frac{{R_i}'(\gamma_{j,s})}{s^{i/2}} \Bigr|
\geq m_{j,s} | \criticalpoint{\gamma}_s - \gamma_{j,s} |,
\end{equation}
which is almost the second claim in \refEquation{eqn:Optimality_gaps_in_main_theorem}. What remains is to remove the dependency on $s$ of $m_{j,s}$. 

To that end, remark that $\max_{\gamma \in [\gamma_l, \gamma_r]} {R_0}''(\gamma) < 0$ by continuity and pointwise negativity of ${R_0}''(\gamma)$. Then, bound
\begin{align}
m_{j,s} 
&
= \min_{\gamma \in [\gamma_l,\gamma_r]} \Bigl\{ - {R_0}''(\gamma) - \sum_{i=1}^j \frac{{R_i}''(\gamma)}{s^{i/2}} \Bigr\}
\nonumber \\ &
\geq \min_{\gamma \in [\gamma_l,\gamma_r]} \{ - {R_0}''(\gamma) \} + \frac{1}{\sqrt{s}} \min_{\gamma \in [\gamma_l,\gamma_r]} \Bigl\{ - \sum_{i=1}^j \frac{{R_i}''(\gamma)}{s^{(i-1)/2}} \Bigr\} 
\\ &
\geq \min_{\gamma \in [\gamma_l,\gamma_r]} \{ - {R_0}''(\gamma) \} - \frac{1}{\sqrt{{s_j}'}} \max_{\gamma \in [\gamma_l,\gamma_r]} \Bigl\{ \sum_{i=1}^j \frac{{R_i}''(\gamma)}{ ({s_j}')^{(i-1)/2}} \Bigr\} 
=: m_{j},
\nonumber 
\end{align}
and increase the value of ${s_j}'$ if necessary to ensure that $m_j > 0$. 

Summarizing, we now have that
$
{M_j}' / s^{(j+1)/2} \geq m_j | \criticalpoint{\gamma}_s - \gamma_{j,s} |
$
for all $s \geq {s_j}'$. Setting $K_j = {M_j}' / m_j$ completes the proof. \QuodEratDemonstrandum

\subsection{Dimensioning under a delay constraint}
\label{sec:Staffing_under_a_delay_constraint}

The approach of \refSection{sec:Optimality_gaps} can also be used to solve delay constrained dimensioning problems. As an example, consider finding
\begin{equation}
\criticalpoint{\gamma}_s = \arg_{ \gamma \in \Gamma } \{ D_s(\gamma) = \epsilon \},
\label{eqn:Staffing_under_a_delay_constraint_problem}
\end{equation}
where $\epsilon \in (0,1)$. Since we have an asymptotic expansion of the form \refEquation{eqn:Expansion_for_Rs} for the delay probability, see \refAppendixSection{sec:Expansions}, instead of solving for $\criticalpoint{\gamma}_s$ directly we can obtain approximations 
$
\gamma_{j,s} = \arg_{\gamma \in \Gamma_j} \{ \sum_{i=0}^j D_i(\gamma) / \allowbreak s^{i/2} = \varepsilon \}
$. The optimality gaps can be calculated using a similar proof technique. 

\begin{corollary}
\label{cor:Delay_probability_gaps}
For $j = 0, 1, \ldots$ there exist finite constants $M_j$, $K_j > 0$ independent of $s$ and $s_j \in \naturalNumbersPlus$ such that for all $s \geq s_j$,
\begin{equation}
| D_s( \gamma_{j,s} ) - \varepsilon |
\leq \frac{M_j}{s^{(j+1)/2}},
\quad
| \gamma_{j,s} - \criticalpoint{\gamma}_s | 
\leq \frac{K_j}{s^{(j+1)/2}}.
\label{eqn:Delay_probability_gaps}
\end{equation}
\end{corollary}

\section{Approaches to asymptotic dimensioning}
\label{sec:Possible_future_directions}

\subsection{Asymptotic revenue maximization with a threshold}
\label{sec:Numerical_results}

We will now consider new approaches to asymptotic dimensioning in the context of optimizing the objective function \refEquation{eqn:Economics_optimization_problem}. We start with considering the revenue maximization problem described in \refSection{sec:Service_systems_with_admission_control_and_revenues} while assuming that the threshold is fixed. This concretely requires us to
maximize $\hat{R}_0(\gamma)$ in \refEquation{eqn:Expression_for_Rshat} over $\gamma$ given a fixed $\eta < \infty$.

The accuracy of our asymptotic expansion as an approximation to the objective function $\hat{R}_s(\gamma)$ is examined in \refFigure{fig:Revenue_together_with_its_first_and_second_order_approximations}, which shows the function $\hat{R}_s(\gamma)$ with its first- and second-order approximations for a system of size $s = 10$. We conclude that both approximations are remarkably accurate for this relatively small system. Near the optimizer $\criticalpoint{\gamma}_s$, the second-order approximation is almost indistinguishable from the objective function. The maximizer of the second-order approximation $\hat{R}_0(\gamma) + \hat{R}_1(\gamma) / \sqrt{s}$ is also closer to the maximizer of $\hat{R}_s(\gamma)$ than the maximizer of the first-order approximation $\hat{R}_0(\gamma)$ is. This illustrates that including higher-order correction terms in the asymptotic expansion indeed reduces the optimality gap. 

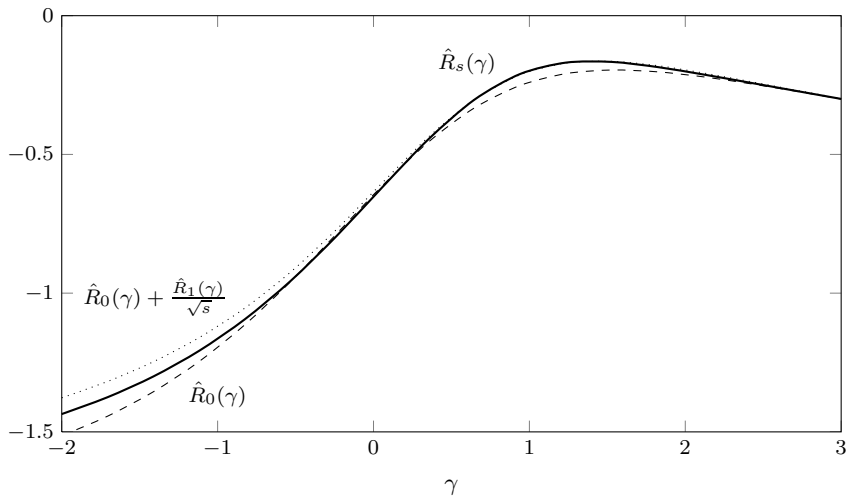
\begin{figure}[!hbtp]
\footnotesize
\centering
\begin{tikzpicture}%
\begin{groupplot}[
	view={0}{90},
	small,
	ymin=0,
	width=0.985\columnwidth, 
	height=0.95*0.61803398875\columnwidth,
	scaled x ticks={real:1.0},
	xtick scale label code/.code={}, 
    group/xlabels at = edge bottom,
    group style = { group size = 1 by 1,
    				xlabels at = edge bottom, 
    				xticklabels at = edge bottom,    				
    				ylabels at = edge right, 
    				horizontal sep = 25pt,
    				vertical sep = 15pt,    				
    				}	
   	]

\nextgroupplot[
				    xmin=-2, xmax=3,
					xtick={-2,-1,0,...,3},
				    ymin=-1.5, ymax=0,
		    		ytick={-1.5,-1,...,0},
		    		xlabel={$\gamma$},
				   ]
				   
\addplot[color=black, thick, smooth] plot coordinates {
(-2., -1.43577) (-1.70588, -1.37503) (-1.41176, -1.30016) (-1.11765, -1.20752) (-0.823529, -1.09332) (-0.529412, -0.954879) (-0.235294, -0.793216) (0.0588235, -0.616957) (0.352941, -0.445016) (0.647059, -0.30283) (0.941176, -0.210045) (1.23529, -0.169083) (1.52941, -0.166667) (1.82353, -0.185254) (2.11765, -0.212128) (2.41176, -0.241195) (2.70588, -0.270588) (3., -0.3) 
};

\addplot[color=black, dashed, smooth] plot coordinates {
(-2., -1.5141) (-1.70588, -1.44343) (-1.41176, -1.3553) (-1.11765, -1.24601) (-0.823529, -1.11301) (-0.529412, -0.95712) (-0.235294, -0.785191) (0.0588235, -0.611512) (0.352941, -0.455063) (0.647059, -0.33248) (0.941176, -0.251297) (1.23529, -0.208727) (1.52941, -0.195786) (1.82353, -0.202419) (2.11765, -0.220537) (2.41176, -0.244758) (2.70588, -0.271953) (3., -0.300485) 
};

\addplot[color=black, dotted, smooth] plot coordinates {
(-2., -1.37721) (-1.70588, -1.31841) (-1.41176, -1.24715) (-1.11765, -1.16007) (-0.823529, -1.0532) (-0.529412, -0.922976) (-0.235294, -0.769401) (0.0588235, -0.600952) (0.352941, -0.437079) (0.647059, -0.302196) (0.941176, -0.213003) (1.23529, -0.170711) (1.52941, -0.164502) (1.82353, -0.180249) (2.11765, -0.206667) (2.41176, -0.236979) (2.70588, -0.268006) (3., -0.298663) 
};


\draw (100,15) node [fill=none, draw=none, text=black, opacity=1, text opacity=1] {$\hat{R}_0(\gamma)$};
\draw (260,135) node [fill=none, draw=none, text=black, opacity=1, text opacity=1] {$\hat{R}_s(\gamma)$};
\draw (60,50) node [fill=none, draw=none, text=black, opacity=1, text opacity=1] {$\hat{R}_0(\gamma) + \frac{\hat{R}_1(\gamma)}{ \sqrt{s}}$};

\end{groupplot}
\end{tikzpicture}
\vspace{-1em}
\caption{The function $\hat{R}_s(\gamma)$ as function of $\gamma$ for $a = 0.1$, $b = 1$, and $\eta = 2$, for a system of size $s = 10$ with its first-order approximation $\hat{R}_0(\gamma)$ (dashed curve) and its second-order approximation $\hat{R}_0(\gamma) + \hat{R}_1(\gamma) / \sqrt{s}$ (dotted curve).}
\label{fig:Revenue_together_with_its_first_and_second_order_approximations}
\end{figure}

The absolute error $|\hat{R}_s(\gamma) - \hat{R}_0(\gamma) - \hat{R}_1(\gamma) / \sqrt{s} |$ is plotted in \refFigure{fig:Absolute_error_in_revenue_function_with_two_terms} as function of $s$ for $\gamma = 2$. A fit is provided which confirms that the asymptotic expansion is indeed accurate up to $\bigO{1/s}$, as suggested by the asymptotic expansion in \refEquation{eqn:Expansion_for_Rs}. The jumps in the data points are caused by rounding in the admission control, since $p_s(k-s) = \indicator{ k - s \leq \lfloor \eta \sqrt{s} \rfloor }$. 

We also examine optimality gaps in \refFigure{fig:Absolute_error_in_revenue_function_with_two_terms}, which shows first- and second-order optimality gaps. Again notice that jumps occur because of the rounding in the control policy. Furthermore, we have provided fits that confirm that the optimality gap is of order $\bigO{1/\sqrt{s}}$ when the asymptotic approximation is of order $\bigO{1}$, and that the optimality is of order $\bigO{1/s}$ when the asymptotic approximation is of order $\bigO{1/\sqrt{s}}$.

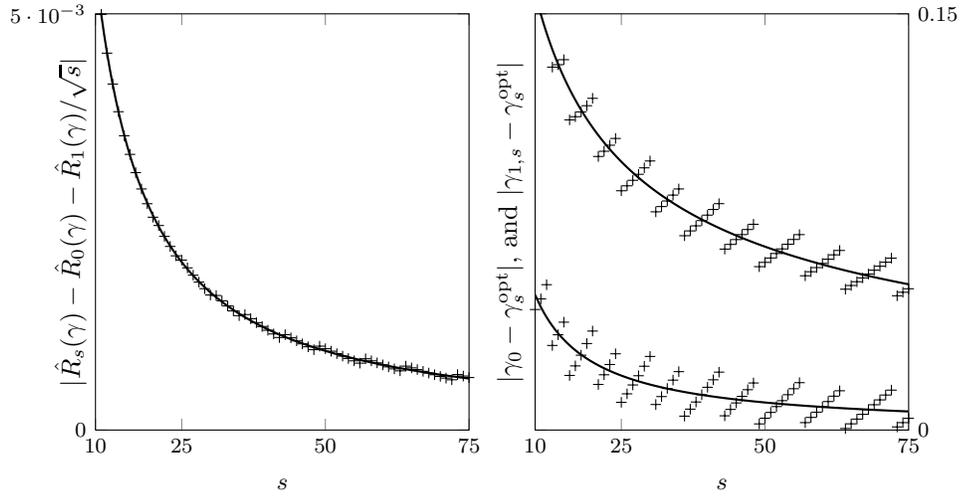
\begin{figure}[!hbtp]
\centering
\begin{tikzpicture}%
\begin{groupplot}[
	view={0}{90},
	small,
	ymin=0,
	width=0.54\columnwidth, 
	height=0.95*0.61803398875\columnwidth,
	scaled x ticks={real:1.0},
	xtick scale label code/.code={}, 
	ylabel absolute, ylabel style={yshift=-0.9cm},
    group/xlabels at = edge bottom,
    group style = { group size = 2 by 1,
    				xlabels at = edge bottom, 
    				xticklabels at = edge bottom,    				
    				ylabels at = edge right, 
    				horizontal sep = 25pt,
    				vertical sep = 15pt,    				
    				}	
   	]

\nextgroupplot[
				    xmin=10, xmax=75,
					xtick={10,25,50,75},
				    ymin=0, ymax=0.005,
		    		ytick={0,0.005},
		    		xlabel={$s$},
	    			ylabel={$| \hat{R}_s(\gamma) - \hat{R}_0(\gamma) - \hat{R}_1(\gamma) / \sqrt{s} |$}
				   ]
				   
\addplot[only marks, mark=+, color=black] plot coordinates {
(10., 0.00555306) (11., 0.00499371) (12., 0.00452777) (13., 0.00415657) (14., 0.00382412) (15., 0.00353495) (16., 0.00331175) (17., 0.0030924) (18., 0.00289604) (19., 0.00271854) (20., 0.00255668) (21., 0.00245708) (22., 0.00232669) (23., 0.00220588) (24., 0.00209326) (25., 0.00204168) (26., 0.00194835) (27., 0.00186043) (28., 0.00177721) (29., 0.00169811) (30., 0.00162262) (31., 0.00161639) (32., 0.00155227) (33., 0.00149072) (34., 0.00143147) (35., 0.00137425) (36., 0.00138717) (37., 0.00133792) (38., 0.00129018) (39., 0.00124381) (40., 0.00119867) (41., 0.00115466) (42., 0.00111166) (43., 0.00114534) (44., 0.00110791) (45., 0.00107126) (46., 0.00103533) (47., 0.00100005) (48., 0.000965381) (49., 0.00100752) (50., 0.000977115) (51., 0.000947185) (52., 0.00091769) (53., 0.000888598) (54., 0.000859878) (55., 0.000831505) (56., 0.000803454) (57., 0.000856326) (58., 0.000831572) (59., 0.000807072) (60., 0.000782809) (61., 0.000758766) (62., 0.00073493) (63., 0.000711286) (64., 0.000767939) (65., 0.000746977) (66., 0.000726167) (67., 0.0007055) (68., 0.000684967) (69., 0.000664559) (70., 0.000644269) (71., 0.00062409) (72., 0.000604016) (73., 0.000666569) (74., 0.000648688) (75., 0.000630888) 
};

\addplot[color=black, thick, smooth] plot coordinates {
(10., 0.00555172) (11., 0.00499446) (12., 0.00453507) (13., 0.00415018) (14., 0.00382327) (15., 0.00354234) (16., 0.00329846) (17., 0.00308486) (18., 0.00289632) (19., 0.00272872) (20., 0.00257882) (21., 0.002444) (22., 0.00232213) (23., 0.00221145) (24., 0.00211051) (25., 0.00201811) (26., 0.00193321) (27., 0.00185497) (28., 0.00178263) (29., 0.00171556) (30., 0.00165322) (31., 0.00159513) (32., 0.00154087) (33., 0.0014901) (34., 0.00144248) (35., 0.00139774) (36., 0.00135563) (37., 0.00131592) (38., 0.00127843) (39., 0.00124296) (40., 0.00120938) (41., 0.00117752) (42., 0.00114727) (43., 0.00111851) (44., 0.00109113) (45., 0.00106504) (46., 0.00104015) (47., 0.00101637) (48., 0.000993651) (49., 0.000971908) (50., 0.000951086) (51., 0.000931127) (52., 0.00091198) (53., 0.000893598) (54., 0.000875936) (55., 0.000858953) (56., 0.000842612) (57., 0.000826878) (58., 0.000811718) (59., 0.000797101) (60., 0.000783) (61., 0.000769389) (62., 0.000756242) (63., 0.000743536) (64., 0.000731251) (65., 0.000719365) (66., 0.000707861) (67., 0.00069672) (68., 0.000685926) (69., 0.000675463) (70., 0.000665317) (71., 0.000655472) (72., 0.000645918) (73., 0.00063664) (74., 0.000627627) (75., 0.000618869) 
};

\nextgroupplot[
				    xmin=10, xmax=75,
					xtick={10,25,50,75},
				    ymin=0, ymax=0.15,
		    		ytick={0,0.15},
		    		xlabel={$s$},
	    			ylabel={$| \gamma_0 - \criticalpoint{\gamma}_s |$, and $| \gamma_{1,s} - \criticalpoint{\gamma}_s |$}
				   ]
				   
\addplot[only marks, mark=+, color=black] plot coordinates {
(10., 0.154885) (11., 0.154577) (12., 0.156002) (13., 0.130803) (14., 0.131665) (15., 0.133473) (16., 0.111787) (17., 0.112932) (18., 0.114662) (19., 0.116893) (20., 0.119555) (21., 0.0985954) (22., 0.100468) (23., 0.102645) (24., 0.105088) (25., 0.08624) (26., 0.0880207) (27., 0.0900044) (28., 0.0921686) (29., 0.0944955) (30., 0.0969717) (31., 0.0786495) (32., 0.0805423) (33., 0.0825529) (34., 0.0846715) (35., 0.0868888) (36., 0.0700323) (37., 0.071771) (38., 0.073592) (39., 0.0754894) (40., 0.0774574) (41., 0.0794907) (42., 0.0815847) (43., 0.065213) (44., 0.0669026) (45., 0.068645) (46., 0.0704367) (47., 0.0722747) (48., 0.0741557) (49., 0.0588826) (50., 0.0604243) (51., 0.0620068) (52., 0.0636256) (53., 0.0652786) (54., 0.0669638) (55., 0.0686793) (56., 0.0704235) (57., 0.055599) (58., 0.0570585) (59., 0.0585442) (60., 0.0600549) (61., 0.0615864) (62., 0.0631428) (63., 0.0647204) (64., 0.0507561) (65., 0.0520928) (66., 0.0534494) (67., 0.0548248) (68., 0.0562183) (69., 0.057629) (70., 0.0590561) (71., 0.0604988) (72., 0.0619535) (73., 0.0483999) (74., 0.0496515) (75., 0.0509173) 
};

\addplot[color=black, thick, smooth] plot coordinates {
(10., 0.157489) (11., 0.149795) (12., 0.143083) (13., 0.137162) (14., 0.131887) (15., 0.127149) (16., 0.122863) (17., 0.11896) (18., 0.115387) (19., 0.1121) (20., 0.109063) (21., 0.106246) (22., 0.103623) (23., 0.101172) (24., 0.098877) (25., 0.0967208) (26., 0.0946901) (27., 0.0927734) (28., 0.0909603) (29., 0.0892417) (30., 0.0876099) (31., 0.0860576) (32., 0.0845787) (33., 0.0831676) (34., 0.0818191) (35., 0.0805289) (36., 0.0792928) (37., 0.0781072) (38., 0.0769687) (39., 0.0758742) (40., 0.0748211) (41., 0.0738067) (42., 0.0728288) (43., 0.0718852) (44., 0.070974) (45., 0.0700933) (46., 0.0692414) (47., 0.0684169) (48., 0.0676183) (49., 0.0668443) (50., 0.0660936) (51., 0.0653651) (52., 0.0646578) (53., 0.0639705) (54., 0.0633024) (55., 0.0626526) (56., 0.0620203) (57., 0.0614048) (58., 0.0608052) (59., 0.0602209) (60., 0.0596513) (61., 0.0590957) (62., 0.0585537) (63., 0.0580246) (64., 0.0575079) (65., 0.0570032) (66., 0.0565101) (67., 0.056028) (68., 0.0555566) (69., 0.0550955) (70., 0.0546443) (71., 0.0542026) (72., 0.0537702) (73., 0.0533467) (74., 0.0529319) (75., 0.0525253) 
};

\addplot[only marks, mark=+, color=black] plot coordinates {
(10., 0.0434385) (11., 0.0472721) (12., 0.0523805) (13., 0.0304878) (14., 0.0343395) (15., 0.0388703) (16., 0.0196782) (17., 0.0231185) (18., 0.0269711) (19., 0.0311761) (20., 0.0356767) (21., 0.0164356) (22., 0.0199209) (23., 0.0236136) (24., 0.0274876) (25., 0.00999086) (26., 0.0130519) (27., 0.0162507) (28., 0.0195703) (29., 0.0229976) (30., 0.0265236) (31., 0.00920449) (32., 0.0120569) (33., 0.0149869) (34., 0.0179873) (35., 0.0210516) (36., 0.00500924) (37., 0.00753148) (38., 0.0101075) (39., 0.0127327) (40., 0.0154032) (41., 0.0181152) (42., 0.0208652) (43., 0.00512823) (44., 0.00742998) (45., 0.00976791) (46., 0.0121371) (47., 0.0145351) (48., 0.01696) (49., 0.00221507) (50., 0.00427015) (51., 0.00635187) (52., 0.00845638) (53., 0.0105822) (54., 0.0127278) (55., 0.014892) (56., 0.0170735) (57., 0.00267567) (58., 0.00455128) (59., 0.00644324) (60., 0.00835058) (61., 0.0102696) (62., 0.0122046) (63., 0.0141523) (64., 0.000549966) (65., 0.00224069) (66., 0.00394368) (67., 0.00565826) (68., 0.00738381) (69., 0.00911973) (70., 0.0108655) (71., 0.0126204) (72., 0.0143813) (73., 0.00112777) (74., 0.00267379) (75., 0.00422843) 
};

\addplot[color=black, thick, smooth] plot coordinates {
(10., 0.0487304) (11., 0.0443212) (12., 0.0406469) (13., 0.0375379) (14., 0.034873) (15., 0.0325634) (16., 0.0305426) (17., 0.0287594) (18., 0.0271744) (19., 0.0257563) (20., 0.0244799) (21., 0.0233251) (22., 0.0222753) (23., 0.0213168) (24., 0.0204382) (25., 0.0196298) (26., 0.0188837) (27., 0.0181928) (28., 0.0175512) (29., 0.0169539) (30., 0.0163964) (31., 0.0158749) (32., 0.015386) (33., 0.0149267) (34., 0.0144944) (35., 0.0140869) (36., 0.0137019) (37., 0.0133378) (38., 0.0129929) (39., 0.0126656) (40., 0.0123547) (41., 0.0120589) (42., 0.0117773) (43., 0.0115087) (44., 0.0112524) (45., 0.0110074) (46., 0.0107731) (47., 0.0105488) (48., 0.0103338) (49., 0.0101276) (50., 0.00992964) (51., 0.00973944) (52., 0.00955655) (53., 0.00938057) (54., 0.0092111) (55., 0.0090478) (56., 0.00889033) (57., 0.00873838) (58., 0.00859168) (59., 0.00844994) (60., 0.00831294) (61., 0.00818042) (62., 0.00805218) (63., 0.00792801) (64., 0.00780772) (65., 0.00769113) (66., 0.00757807) (67., 0.00746839) (68., 0.00736194) (69., 0.00725857) (70., 0.00715815) (71., 0.00706056) (72., 0.00696569) (73., 0.00687341) (74., 0.00678362) (75., 0.00669624) 
};
\end{groupplot}
\end{tikzpicture}
\vspace{-1em}
\caption{(left) The absolute error $| \hat{R}_s(\gamma) - \hat{R}_0(\gamma) - \hat{R}_1(\gamma) / \sqrt{s} |$ as function of $s$ for $a = 0.1$, $b = 1$, and $\eta = 2$, for a critically scaled system with $\gamma = 2$. Plotted also is the curve $e(s) = c_1 + c_2 / s^{c_3}$ with fit parameters $c_1 = 4.0 \cdot 10^{-5}$, $c_2 = 7.2 \cdot 10^{-2}$, and $c_3 = 1.1$ (continuous). (right) The top data points give the optimality gap $| \gamma_0 - \criticalpoint{\gamma}_s |$, and the bottom data points $| \gamma_{1,s} - \criticalpoint{\gamma}_s |$. The top fit is for $e(s) = c_1 + c_2 / \sqrt{s}$ with $c_1 = 2.3 \cdot 10^{-4}$ and $c_2 = 4.9 \cdot 10^{-1}$, the bottom $e(s) = c_1 + c_2 / s$ with $c_1 = -1.0 \cdot 10^{-5}$ and $c_2 = 1.3 \cdot 10^{-3}$.}
\label{fig:Absolute_error_in_revenue_function_with_two_terms}
\end{figure}

\subsection{Joint dimensioning and admission control}
\label{sec:Joint_staffing_and_admission_control}

We now introduce joint dimensioning and admission control, which has not been studied in the \gls{QED} literature. In \cite{sanders_optimal_2014}, it is proven that the admission control policy that maximizes the system's revenue rate has a threshold structure, and that for fixed $\gamma < \infty$ there exists an optimal threshold level $\criticalpoint{\eta}$. In \refSection{sec:Numerical_results}, we fixed $\eta < \infty$ and searched for $\criticalpoint{\gamma}$. Now consider their joint optimization, that is to find
\begin{equation}
(\criticalpoint{\gamma}, \criticalpoint{\eta} ) = \arg \max_{ \gamma \in \realNumbers, \, \eta \geq 0 } \{ \hat{R}_0(\gamma, \eta) \}.
\label{eqn:Joint_staffing_and_admission_control_problem}
\end{equation}

We now illustrate numerically that joint dimensioning and admission control provides important improvements compared to optimizing over $\gamma$ only as in \refSection{sec:Numerical_results}. \refTable{table:Optimal_threshold_hedge_pair} displays solution pairs $(\criticalpoint{\gamma},\criticalpoint{\eta})$ to the maximization problem in \eqref{eqn:Joint_staffing_and_admission_control_problem} for various ratios $r_1 = a/(a+d)$ and $r_2 = (a+d)/b$. Note that these two ratios are sufficient to describe all possible optimization problems, i.e.\ that occur for different $a$, $b$, and $d$, which can for example be seen by dividing the objective function in \refEquation{eqn:Asymptotic_limit_of_economics_optimizaton_problem} by $d$ and then noting that $(a+d)/d = r_2/(r_2 - r_1 r_2)$ and $b/d = 1/(r_2 - r_1 r_2)$. From \refTable{table:Optimal_threshold_hedge_pair}, we see that the optimization problem is well-posed, since nondegenerate optimal pairs exist. 

\refTable{table:Improvements_from_considering_joint_staffing_and_admission_control} contains the percentage improvements that can be achieved by solving the joint dimensioning and admission control problem, compared to the classical approach of finding $\criticalpoint{\gamma}_\infty = \arg \max_{\gamma \geq 0} \hat{R}_s(\gamma, \eta = \infty)$. We should note that this concerns the maximization of the second-order term in \refEquation{eqn:Economics_optimization_problem}, the leading order cannot be influenced through optimization. Note also that $\criticalpoint{\gamma}$ can become negative.

\begin{sidewaystable}[!hbtp]
\centering
\caption{The optimal threshold, hedge pair $(\gamma^{\mathrm{opt}},\eta^{\mathrm{opt}})$ for different ratios of $a/(a+d)$ and $(a+d)/b$.}
\scriptsize
\label{table:Optimal_threshold_hedge_pair}
\begin{tabular}{c|ccccccccc}
\toprule 
$a/(a+d)$ & \multicolumn{9}{c}{$(a+d)/b$} \\
& 1/5 & 1/4 & 1/3 & 1/2 & 1 & 2 & 3 & 4 & 5 \\
\midrule
0.1 & (1.9, 0.4) & (1.9, 0.5) & (1.9, 0.6) & (1.8, 0.9) & (1.6, 1.6) & (1.4, 2.9) & (1.2, 3.9) & (1.1, 4.7) & (1.1, 5.5) \\ 
0.2 & (1.5, 0.3) & (1.5, 0.4) & (1.5, 0.5) & (1.4, 0.7) & (1.3, 1.3) & (1.1, 2.3) & (1.0, 3.1) & (0.9, 3.8) & (0.8, 4.4) \\ 
0.3 & (1.2, 0.3) & (1.2, 0.3) & (1.1, 0.4) & (1.1, 0.6) & (1.0, 1.1) & (0.8, 2.0) & (0.7, 2.6) & (0.7, 3.2) & (0.6, 3.7) \\ 
0.4 & (0.9, 0.2) & (0.8, 0.3) & (0.8, 0.4) & (0.8, 0.5) & (0.7, 1.0) & (0.6, 1.7) & (0.6, 2.3) & (0.5, 2.8) & (0.5, 3.2) \\ 
0.5 & (0.5, 0.2) & (0.5, 0.2) & (0.5, 0.3) & (0.5, 0.5) & (0.5, 0.8) & (0.4, 1.5) & (0.4, 2.0) & (0.3, 2.4) & (0.3, 2.8) \\ 
0.6 & (0.2, 0.2) & (0.2, 0.2) & (0.2, 0.3) & (0.2, 0.4) & (0.2, 0.7) & (0.2, 1.2) & (0.2, 1.7) & (0.1, 2.0) & (0.1, 2.4) \\ 
0.7 & (-0.3, 0.1) & (-0.3, 0.2) & (-0.3, 0.2) & (-0.2, 0.3) & (-0.2, 0.6) & (-0.2, 1.0) & (-0.1, 1.4) & (-0.1, 1.7) & (-0.1, 2.0) \\ 
0.8 & (-0.9, 0.1) & (-0.9, 0.1) & (-0.9, 0.2) & (-0.8, 0.2) & (-0.7, 0.4) & (-0.6, 0.8) & (-0.5, 1.1) & (-0.5, 1.3) & (-0.5, 1.6) \\ 
0.9 & (-2.1, 0.1) & (-2.1, 0.1) & (-2.1, 0.1) & (-2.0, 0.2) & (-1.8, 0.3) & (-1.5, 0.5) & (-1.4, 0.7) & (-1.3, 0.9) & (-1.2, 1.0) \\ 
\bottomrule
\end{tabular}

\centering
\caption{Pairs $( \gamma^{\mathrm{opt}} / \gamma^{\mathrm{opt}}_\infty, 100 \cdot ( R^{\mathrm{opt}} - R_{\infty}^{\mathrm{opt}} ) / | R_{\infty}^{\mathrm{opt}} | )$ for ratios $a/(a+d)$ and $(a+d)/b$.}
\scriptsize
\label{table:Improvements_from_considering_joint_staffing_and_admission_control}
\begin{tabular}{c|ccccccccc}
\toprule 
$a/(a+d)$ & \multicolumn{9}{c}{$(a+d)/b$} \\
& 1/5 & 1/4 & 1/3 & 1/2 & 1 & 2 & 3 & 4 & 5 \\
\midrule
0.1 & (0.9, 7) & (0.9, 6) & (0.9, 5) & (0.9, 3) & (1.0, 1) & (1.0, 0) & (1.0, 0) & (1.0, 0) & (1.0, 0) \\ 
0.2 & (0.8, 13) & (0.8, 11) & (0.8, 9) & (0.8, 7) & (0.9, 4) & (0.9, 2) & (0.9, 1) & (1.0, 1) & (1.0, 1) \\ 
0.3 & (0.6, 18) & (0.7, 16) & (0.7, 14) & (0.7, 11) & (0.8, 7) & (0.8, 4) & (0.9, 3) & (0.9, 3) & (0.9, 2) \\ 
0.4 & (0.5, 23) & (0.5, 22) & (0.5, 19) & (0.6, 16) & (0.6, 11) & (0.7, 8) & (0.7, 7) & (0.7, 6) & (0.7, 5) \\ 
0.5 & (0.3, 29) & (0.3, 28) & (0.4, 25) & (0.4, 22) & (0.4, 17) & (0.5, 13) & (0.5, 11) & (0.5, 10) & (0.5, 9) \\ 
0.6 & (0.1, 36) & (0.1, 34) & (0.1, 32) & (0.1, 29) & (0.2, 23) & (0.2, 19) & (0.2, 17) & (0.2, 16) & (0.2, 15) \\ 
0.7 & (-0.2, 44) & (-0.2, 42) & (-0.2, 40) & (-0.2, 37) & (-0.2, 31) & (-0.2, 27) & (-0.2, 25) & (-0.2, 23) & (-0.2, 22) \\ 
0.8 & (-0.6, 54) & (-0.6, 52) & (-0.7, 50) & (-0.7, 47) & (-0.8, 42) & (-0.9, 38) & (-0.9, 36) & (-0.9, 34) & (-1.0, 33) \\ 
0.9 & (-1.5, 67) & (-1.5, 65) & (-1.6, 64) & (-1.8, 62) & (-2.0, 58) & (-2.3, 54) & (-2.5, 52) & (-2.6, 51) & (-2.6, 50) \\ 
\bottomrule
\end{tabular}
\end{sidewaystable}

\subsection{Refined dimensioning}

The \gls{QED} refinements can also be used to derive refined dimensioning levels. The idea is that a higher-order asymptotic solution $\gamma_{1,s}$ can be expressed as a function of the lower-order asymptotic solution $\gamma_{0}$. To see this, consider the following representation of $\gamma_{1,s}$ in the context of dimensioning under a delay constraint as discussed in \refSection{sec:Staffing_under_a_delay_constraint}.

\begin{theorem}
\label{thm:Refined_staffing}
For sufficiently large $s$, the first-order solution is given as
$
\gamma_{1,s} = \gamma_{0} + \bar{\gamma}_{0}
$,
where 
\begin{equation}
\bar{\gamma}_0
= \sum_{n=1}^\infty \frac{(-1)^n}{(n-1)!} \Bigl( \frac{\d{}}{\d{\gamma}} \Bigr)^{n-1} \Bigl[ ( A'(\gamma) + E'(\gamma) ) \Bigl( \frac{\gamma}{A(\gamma)} \Bigr)^{n+1} (E(\gamma))^n \Bigr]_{\gamma = 0},
\label{eqn:Refinement_in_delay_probability_case}
\end{equation} 
and for small $\gamma$ the auxiliary functions are defined as
$A(\gamma) 
= D_0( \gamma + D_0^{\gets}(\varepsilon) ) - \varepsilon$,
and
$
E(\gamma) 
= D_1( \gamma + D_0^{\gets}(\varepsilon) ) / \sqrt{s}
$.
\end{theorem}

\begin{proof}
We mimick the standard proof of Lagrange's inversion theorem. Both $A$ and $E$ are analytic in the neighborhood of $\gamma = 0$ by analyticity of $B_s(\gamma)$ around $\gamma = 0$ and analyticity of $F_1(\gamma)$ in $\mathrm{Re} \{ \gamma \} > \gamma^{\min}$, with $\gamma^{\min} < 0$ by assumption. Also, $A(0) = 0$, $A'(0) \neq 0$. 

By taking $s$ sufficiently large, say $s \geq s_0$, we can arrange that there is an $r > 0$ such that $A'(0) + G'(0) \neq 0$ and $| E(\gamma) / A(\gamma) \allowbreak \leq 1/2$ for $|\gamma| = r$ and $s \geq s_0$, while $\gamma = 0$ is the only zero of $F(\gamma)$ in $|\gamma| \leq r$. By Rouch\'{e}'s theorem, for any $s \geq s_0$ there is a unique $\bar{\gamma}_0$ in $|\gamma| \leq r$ such that $A(\bar{\gamma}_0) + E(\bar{\gamma}_0) = 0$. Thus the solution $\gamma_{1,s}$ near $\gamma_0$ of the equation $D_0(\gamma) + D_1(\gamma) / \sqrt{s} = \varepsilon$ is given as $\gamma_{1,s} = \gamma_0 + \bar{\gamma}_0$. By Cauchy's theorem, there is for $\bar{\gamma}_0$ the integral representation 
\begin{equation}
\bar{\gamma}_0 = \frac{1}{2 \pi \im} \int_{|\gamma|=r} \frac{\gamma ( A'(\gamma) + E'(\gamma) )}{A(\gamma) + E(\gamma)} \d{\gamma}.
\end{equation}
Since we have $0 \neq |A(\gamma)| \geq 2|E(\gamma)|$ on $|\gamma| = r$, it follows that
\begin{equation}
\bar{\gamma}_0 = \sum_{n=0}^\infty \frac{(-1)^n}{2 \pi \im} \int_{|\gamma|=r} \gamma( A'(\gamma) + E'(\gamma) ) \frac{ (E(\gamma))^n }{ (A(\gamma))^{n+1} } \d{\gamma}.
\end{equation}
Due to analyticity of $\gamma / A(\gamma)$, the term with $n = 0$ vanishes. For $n = 1, 2, \ldots$, we furthermore have that
\begin{align}
&
\frac{1}{2 \pi \im} \int_{|\gamma|=r} \gamma( A'(\gamma) + E'(\gamma) ) \frac{ (E(\gamma))^n }{ (A(\gamma))^{n+1} } \d{\gamma}
\\ &
= \frac{1}{2 \pi \im} \int_{|\gamma|=r} \frac{1}{\gamma^n} ( A'(\gamma) + E'(\gamma) ) \Bigl( \frac{\gamma}{A(\gamma)} \Bigr)^{n+1} (E(\gamma))^n \d{\gamma}
\nonumber \\ &
= \frac{1}{(n-1)!} \Bigl( \frac{\d{}}{\d{\gamma}} \Bigr)^{n-1} \Bigl[ ( A'(\gamma) + E'(\gamma) ) \Bigl( \frac{\gamma}{A(\gamma)} \Bigr)^{n+1} (E(\gamma))^n \Bigr]_{\gamma = 0}.
\nonumber 
\end{align}
This is the result in \refEquation{eqn:Refinement_in_delay_probability_case}, and concludes the proof.
\end{proof}

When using the first two terms in \refEquation{eqn:Refinement_in_delay_probability_case}, we get for $\bar{\gamma}_0$ the approximation
\begin{equation}
- \frac{E}{A'} + \frac{EE'}{(A')^2} - \frac{A''E^2}{2(A')^3} - \frac{3A''E'E^2}{(A')^4} + \frac{E''E^2}{(A')^3} + \frac{2E(E')^2}{(A')^3},
\end{equation}
with all functions evaluated at $\gamma = 0$. The first term gives an $s^{-1/2}$-correction of $\gamma_0$, the second term gives an $s^{-1}$-contribution, and other terms give contributions of $\bigO{s^{-3/2}}$ or smaller. Summarizing, we thus have that
\begin{equation}
\gamma_{1,s} 
= \gamma_0 + \bar{\gamma}_0 
= \gamma_0 - \frac{1}{\sqrt{s}} \frac{D_1(\gamma_0)}{{D_0}'(\gamma_0)} + \bigObig{ \frac{1}{s} }. 
\end{equation}

\bibliographystyle{abbrv}
\bibliography{G:/SOR/Bibliography/Bibliography}

\section*{Acknowledgments}

This research was financially supported by The Netherlands Organization for Scientific Research (NWO) in the framework of the TOP-GO program and by an ERC Starting Grant.

\appendix

\section{Higher-order terms in asymptotic expansions}
\label{sec:Expansions}

We now provide closed-form expressions for the asymptotic expansion in \refEquation{eqn:Expansion_for_Rs}. We drop the dependence on $\gamma$ for notational convenience, and prove the following result.

\begin{theorem}
\label{thm:Expansions}
As $s \to \infty$,
$
( R_s - n_s ) / q_s 
= R_0 + R_1 / \sqrt{s} + \bigO{ 1 / s }
$,
where
\begin{equation}
R_0 
= \frac{ W_0^{\mathrm{L}} + W_0^{\mathrm{R}} }{ B_0 + F_0 },
\enskip
R_1 
= \frac{ W_1^{\mathrm{L}} + W_1^{\mathrm{R}} }{ B_0 + F_0 } - \frac{ ( W_0^{\mathrm{L}} + W_0^{\mathrm{R}} ) ( B_1 + F_1 ) }{ ( B_0 + F_0 )^2 },
\end{equation}
and
\begin{gather}
W_0^{\mathrm{L}} 
= \int_{-\infty}^0 r(x) \e{-\frac{1}{2} x^2 - \gamma x } \d{x}, 
\nonumber \\ 
W_1^{\mathrm{L}} 
= \tfrac{1}{2} \int_{-\infty}^0 ( \tfrac{1}{3} x^3 - ( 1 + \gamma^2 ) x ) r(x) \e{-\frac{1}{2} x^2 - \gamma x } \d{x} + r(0),
\nonumber \\
W_0^{\mathrm{R}} 
= \int_0^\infty r(x) f(x) \e{-\gamma x} \d{x},
\nonumber \\
W_1^{\mathrm{R}} 
= - \tfrac{1}{2} \gamma^2 \int_0^\infty x r(x) f(x) \e{-\gamma x} \d{x} - \tfrac{1}{2} r(0) f(0),
\label{eqn:Explicit_formulae_for_expansion_coefficients_of_WsL_and_WsR}
\end{gather}
as well as
\begin{gather}
B_0 
= \frac{\Phi(\gamma)}{\phi(\gamma)},
\enskip 
B_1
= \tfrac{1}{3} \Bigl( 2 + \gamma^2 + \gamma^3 \frac{\Phi(\gamma)}{\phi(\gamma)} \Bigr),
\label{eqn:Explicit_formulae_for_expansion_coefficients_of_Bs_and_Fs}
\\
F_0 
= \int_0^\infty f(x) \e{-\gamma x} \d{x},
\enskip
F_1 
= - \tfrac{1}{2} \gamma^2 \int_0^\infty x f(x) \e{-\gamma x} \d{x} - \tfrac{1}{2} f(0).
\nonumber
\end{gather}
\end{theorem}

\begin{proof}
Note after substituting \refEquation{eqn:Equilibrium_distribution} into $R_s  = \sum_{k=0}^\infty r_s(k) \pi_s(k)$, that asymptotically
\begin{equation}
\frac{R_s - n_s}{q_s}
= \frac{ \sum_{k=0}^s r\bigl( \frac{k-s}{\sqrt{s}} \bigr) \frac{(s\rho)^k}{k!} + \frac{(s\rho)^s}{s!} \sum_{k=s+1}^\infty r\bigl( \frac{k-s}{\sqrt{s}} \bigr) \rho^{k-s} f\bigl( \frac{k-s}{\sqrt{s}} \bigr) }{ \sum_{k=0}^s \frac{(s\rho)^k}{k!} + \frac{(s\rho)^s}{s!} \sum_{k=s+1}^\infty \rho^{k-s} f\bigl( \frac{k-s}{\sqrt{s}} \bigr) }.
\end{equation}
Dividing by the factor $(s\rho)^s / s!$, we obtain the form
\begin{equation}
\frac{R_s - n_s}{q_s}
= \frac{ W_s^{\mathrm{L}} + W_s^{\mathrm{R}} }{ B_s^{-1} + F_s },
\label{eqn:Equation_for_Rs_divided_into_asymptotic_terms}
\end{equation}
where we have introduced notation for the Erlang B formula,
$
B_s(\rho) = ( (s\rho)^s / s! ) / \allowbreak ( \sum_{k=0}^s (s\rho)^k / k! )
$,
and we have defined 
$
F_s
= \sum_{n=0}^\infty \rho^{n+1} f( (n+1) / \sqrt{s} )
$, 
$
W_s^{\mathrm{L}}
= \sum_{k=0}^s r( (k-s) / \sqrt{s} ) ( s! (s\rho)^{k-s} ) / k!
$,
and
$
W_s^{\mathrm{R}}
= \sum_{n=0}^\infty r( (n+1) / \sqrt{s} ) \rho^{n+1} f( (n+1) / \sqrt{s} )
$.
In \cite{janssen_scaled_2013,sanders_optimal_2014}, it is proven using Jagerman's asymptotic expansions \cite{jagerman_properties_1974} for Erlang B's formula that asymptotically,
\begin{gather}
W_s^{\mathrm{L}} 
= \sqrt{s} W_0^{\mathrm{L}} + W_1^{\mathrm{L}} + \bigObig{ \frac{1}{\sqrt{s}} },
\enskip
W_s^{\mathrm{R}} 
= \sqrt{s} W_0^{\mathrm{R}} + W_1^{\mathrm{R}} + \bigObig{ \frac{1}{\sqrt{s}} },
\nonumber \\
B_s^{-1} 
= \sqrt{s} B_0 + B_1 + \bigObig{ \frac{1}{\sqrt{s}} },
\enskip
F_s 
= \sqrt{s} F_0 + F_1 + \bigObig{ \frac{1}{\sqrt{s}} },
\nonumber
\end{gather}
with the coefficients as given in \refEquation{eqn:Explicit_formulae_for_expansion_coefficients_of_WsL_and_WsR} and \refEquation{eqn:Explicit_formulae_for_expansion_coefficients_of_Bs_and_Fs}. After substituting these asymptotic expansions into \refEquation{eqn:Equation_for_Rs_divided_into_asymptotic_terms}, we obtain
\begin{equation}
\frac{R_s - n_s}{q_s}
= \frac{ W_0^{\mathrm{L}} + W_0^{\mathrm{R}} }{ B_0 + F_0 } \cdot \frac{ 1 + \frac{1}{\sqrt{s}} ( W_1^{\mathrm{L}} + W_1^{\mathrm{R}} ) / ( W_0^{\mathrm{L}} + W_0^{\mathrm{R}} ) + \bigO{ \frac{1}{s} } }{ 1 + \frac{1}{\sqrt{s}} ( B_1 + F_1  ) / ( B_0 + F_0 ) + \bigO{ \frac{1}{s} } }.
\nonumber
\end{equation}
By then utilizing the Taylor expansion $1/(1+x) = 1 - x + \bigO{x^2}$, we obtain the result.
\end{proof}

\paragraph{Delay probability}

The delay probability $D_s = \sum_{k=s}^\infty \pi_s(k)$ can be represented by $r_s(k) = \indicator{ k \geq s }$, recall \refEquation{eqn:Weight_function_Rs}. This corresponds asymptotically to $r(x) = \indicator{ x \geq 0 }$. introduce for convenience $\mathcal{L} = \int_0^\infty f(x) \e{-\gamma x} \d{x}$. Then,
$
W_0^{\mathrm{L}} 
= W_1^{\mathrm{L}} 
= 0$, 
$
W_0^{\mathrm{R}} 
= \mathcal{L}
$,
$
W_1^{\mathrm{R}} 
= \tfrac{1}{2} \gamma^2 \mathcal{L}'
$,
$
F_0 
= \mathcal{L}
$,
and
$
F_1 
= \tfrac{1}{2} \gamma^2 \mathcal{L}' - \tfrac{1}{2}
$.
It follows that 
$
D_0 
= \mathcal{L} / ( \Phi(\gamma) / \phi(\gamma) + \mathcal{L} )
$
and
\begin{equation}
D_1 
= \frac{ \tfrac{1}{2} \gamma^2 \mathcal{L}' }{ \frac{\Phi(\gamma)}{\phi(\gamma)} + \mathcal{L} } - \frac{ \mathcal{L} ( \tfrac{1}{3} \bigl( 2 + \gamma^2 + \gamma^3 \frac{\Phi(\gamma)}{\phi(\gamma)} \bigr) + \tfrac{1}{2} \gamma^2 \mathcal{L}' - \tfrac{1}{2} ) }{ ( \frac{\Phi(\gamma)}{\phi(\gamma)} + \mathcal{L} )^2 }.
\label{eqn:Expansions_for_the_delay_probability}
\end{equation}

\paragraph{Queue length}

The mean queue length $Q_s = \sum_{k=s}^\infty (k-s) \pi_s(k)$ can be represented by $r_s(k) = (k-s) \indicator{ k \geq s }$. Scaling so that $Q_s / \sqrt{s} =  \sum_{k=s}^\infty ( (k-s) / \sqrt{s} ) \pi_s(k)$, we see that the revenue structure can asymptotically be related to the revenue profile $r(x) = x \indicator{ x \geq 0 }$. Therefore,
$
W_0^{\mathrm{L}} 
= W_1^{\mathrm{L}} 
= 0
$, 
$
W_0^{\mathrm{R}} 
= -\mathcal{L}'
$,
$
W_1^{\mathrm{R}} 
= - \tfrac{1}{2} \gamma^2 \mathcal{L}''
$,
$
F_0 
= \mathcal{L}
$,
and
$
F_1 
= \tfrac{1}{2} \gamma^2 \mathcal{L}' - \tfrac{1}{2}
$.
Thus $Q_s / \sqrt{s} = Q_0 + Q_1 / \sqrt{s} + \bigO{ 1/s }$ with
$
Q_0 
= - \mathcal{L}' / ( \Phi(\gamma) / \phi(\gamma) + \mathcal{L} )
$
and
\begin{equation}
Q_1 
= - \frac{ \tfrac{1}{2} \gamma^2 \mathcal{L}' }{ \frac{\Phi(\gamma)}{\phi(\gamma)} + \mathcal{L} } + \frac{ \mathcal{L}' ( \tfrac{1}{3} \bigl( 2 + \gamma^2 + \gamma^3 \frac{\Phi(\gamma)}{\phi(\gamma)} \bigr) + \tfrac{1}{2} \gamma^2 \mathcal{L}' - \tfrac{1}{2} ) }{ ( \frac{\Phi(\gamma)}{\phi(\gamma)} + \mathcal{L} )^2 }.
\label{eqn:Expansions_for_the_queue_size}
\end{equation}

\end{document}